\documentclass[notitlepage,11pt,reqno]{amsart}
\usepackage{amsmath, amssymb, amsfonts, amstext, amsthm, color,tocvsec2}
\usepackage[breaklinks=true]{hyperref}
\usepackage[mathscr]{euscript}
\usepackage{enumerate}
\usepackage[letterpaper]{geometry}
\usepackage{float}
\usepackage{pgf,tikz,ulem}

\newtheorem{thm}{Theorem}[section]
\newtheorem{defn}[thm]{Definition}
\newtheorem{lem}[thm]{Lemma}
\newtheorem{cor}[thm]{Corollary}
\newtheorem{prop}[thm]{Proposition}
\newtheorem{ex}[thm]{Example}
\newtheorem{rem}[thm]{Remark}

\DeclareMathOperator{\qec}{QEC}
\newcommand{\dg}{\ensuremath{D_G}}
\def\be{\begin{equation}}
\def\ee{\end{equation}}
\def\bea{\begin{eqnarray}}
\def\eea{\end{eqnarray}}
\numberwithin{equation}{section}
\begin{document}
	
	\title[Quadratic Embedding Constants of graphs: Bounds and distance spectra]{Quadratic Embedding Constants of graphs: Bounds and distance spectra}

	\author{Projesh Nath Choudhury}
	\address[P.N.~Choudhury]{Discipline of Mathematics, Indian Institute of Technology Gandhinagar,
		Palaj, Gandhinagar 382355, India}
	\email{\tt projeshnc@iitgn.ac.in}

	\author{Raju Nandi}
	\address[R.~Nandi]{Discipline of Mathematics, Indian Institute of Technology Gandhinagar,
		Palaj, Gandhinagar 382355, India}
	\email{\tt raju.n@iitgn.ac.in, rajunandirkm@gmail.com}
	
	\date{\today}
	
	\begin{abstract}
		The quadratic embedding constant (QEC) of a finite, simple, connected graph $G$ is the maximum of the quadratic form of the distance matrix of $G$ on the subset of the unit sphere orthogonal to the all-ones vector. The study of these QECs was motivated by the classical work of Schoenberg on quadratic embedding of metric spaces [\textit{Ann.\ of\ Math.}, 1935] and [\textit{Trans.\ Amer.\ Math.\ Soc.}, 1938].  In this article, we provide sharp upper and  lower bounds for the  QEC of trees. We next explore the relation between  distance spectra and quadratic embedding constants of graphs -- and show two further results: $(i)$ We show that the quadratic embedding constant of a graph is zero if and only if its second largest distance eigenvalue is zero. $(ii)$ We identify a new subclass of nonsingular graphs whose  QEC is the second largest distance eigenvalue. Finally, we show that the QEC of the cluster of an arbitrary graph $G$ with either a complete or star graph can be computed in terms of the QEC of $G$. As an application of this result, we provide new families of examples of graphs of QE class.
	\end{abstract}
	
	

%
%
%
%
%
%
%
 \maketitle
 
\section{Introduction and main results}
\textit{Throughout this paper, all graphs are assumed  to be finite, simple, connected, and unweighted. Given a graph $G=(V,E)$, we define $C(V):$ the space of all $\mathbb{R}$-valued functions on $V$, $e:$ the constant function in $C(V)$ taking value $1$,  and $J :$ matrix with all entries one. The canonical inner product on $C(V)\cong \mathbb{R}^{|V|}$ will be denoted by  $\langle .,. \rangle$. Also, given a graph $G$ with vertex set $V$, $D_G$ denotes its distance matrix, with the $(v,w)$ entry given by the length of a shortest path connecting $v \neq w \in V$, and $(D_G)_{vv} = 0\ \forall v \in V$. The eigenvalues of $D_G$ are arranged as $\lambda _1 (G)>\lambda _2(G) \geq \lambda _3(G)\geq \cdots \geq \lambda _{|V|}(G).$}
\medskip

The aim of this work is to study a graph-theoretic constant arising from Euclidean distance geometry.  A well-studied problem in distance geometry asks whether given a graph $G=(V,E)$ there exists a function $\psi$ from $V$ into a Euclidean space $\mathcal{H}$  such that $d(i,j)=\  \parallel \psi (i) - \psi (j) \parallel ^2$ for all $i,j\in V$, where $\parallel . \parallel $ denotes the norm of $\mathcal{H}$. Such a map $\psi: V \rightarrow \mathcal{H}$ is called a {\it quadratic embedding } of $G$ and if $G$ admits a quadratic embedding, we say that $G$ is of \textit{QE class}. Graphs of QE class have rich applications in quantum probability and non-commutative harmonic analysis \cite{mbozejko, bozejko, graham, haagerup,hora,nobata,obata}.
 In \cite{schoenberg}, Schoenberg showed a striking characterization of the $QE$ class:
\begin{thm}\label{schointro}
Let $G=(V,E)$ be a  graph with at least two vertices. Then $G$ is of QE class if and only if $\dg$ is conditionally negative definite, that is  $\langle f,\dg f \rangle \leq 0$  for all\  $f\in C(V)$\ \ with\  $\langle e,f \rangle = 0$.
\end{thm}
	
Motivated by Schoenberg's result, the authors in \cite{zakiyyah} introduced the notion of {\it quadratic embedding constant} of $G$ denoted and defined as
\begin{center}
$\qec(G):= \max\{\langle f,\dg f \rangle ;\  f\in C(V),\  \langle f,f \rangle = 1,\  \langle e,f \rangle = 0 \}.$
\end{center}
In particular, Schoenberg's result says a graph $G$ is of QE class if and only if $\qec (G)\leq 0$. $\qec(G)$ has since been studied by several authors  \cite{baskoro,lou,wmlotkowski,mo,zakiyyah}. For instance, the quadratic embedding constants of $K_n$, $C_n$ and complete multipartite graphs are given in \cite{bata,zakiyyah}. E.g., $\qec(K_{m,m})=m-2$; in particular $K_{m,m}$ is not of QE class for $m\geq 3$.

 We now consider trees. For a path graph $P_n$ on $n$ vertices, the authors in \cite{mo} derived a lower and upper bound for $\qec (P_n)$. M{\l}otkowski \cite{wmlotkowski} recently extended this work by providing the exact value: $\qec (P_n)=\frac{-1}{1+\cos (\pi/n)}$. In \cite{zakiyyah}, authors studied quadratic embedding of graphs and proved that any tree $T$ is of QE class. Also, from \cite{wmlotkowski, zakiyyah}, $\qec (P_n) < \qec (K_{1,n-1})=-2/n$, for $n\geq 4$. Now our first main result shows that the quadratic embedding constants of any tree on $n\geq 4$ vertices lies between $\qec (P_n)$ and $\qec (K_{1,n-1})$. 
\begin{thm}\label{ulbound}
	Let $T_n$ be a tree on $n\geq 4$ vertices. Then
	$\qec(P_n)\leq \qec(T_n)\leq \qec(K_{1,n-1})$. Moreover, the left side equality holds if and only if $T_n=P_n$  while the right hand equality holds if and only if $T_n=K_{1,{n-1}}$.
\end{thm}

\begin{rem}
  In \cite{mo}, the authors showed that for any tree $T_n$ on $n\geq 3$ vertices
	\begin{equation}\label{tree}
		\qec(T_n)\leq -\frac{2}{2n-3},
	\end{equation}
	and this inequality is strict for $n=4,5$ (see \cite[Section 5]{zakiyyah}). A natural problem, which is mentioned in \cite{baskoro}, is to find tight upper and lower  bound for $\qec(T_n)$, where $T_n$ runs over all trees with $n$ vertices. 
	 By \cite[Theorem 2.8]{zakiyyah}, $\qec(K_{1,n-1})=-\frac{2}{n}<-\frac{2}{2n-3}$. Hence our first main result answers this question.
\end{rem}

An interesting topic in spectral graph theory is exploring the relation between distance spectra and the quadratic embedding constants of graphs. The distance matrix of a connected graph is an irreducible, nonnegative, symmetric matrix. By the Perron--Frobenius Theorem, its largest eigenvalue is simple. By the min--max principle for eigenvalues of symmetric matrices \cite[Theorem 4.2.6]{horn},
\begin{equation}\label{minmax}
\lambda_2 (G)\leq \qec(G)<\lambda_1(G).
\end{equation}   This leads us to a natural question that has been pursued in the literature: \textit{Characterize the family of graphs that satisfy $\lambda_2(G)=QEC(G)$}. It is well known that this equality holds for distance regular graphs and transmission regular graphs \cite{baskoro}.  In [17], the author showed that $\lambda_2(G)=QEC(G)$ holds for path graphs with an even number of vertices. In this article, we show that this holds for unicyclic graphs with an even cycle, cacti with at least one even cycle, linear hexagonal chains and wheel graphs with odd vertices. In fact, we prove a stronger result:

\begin{thm}\label{sns}
	Let $G$ be a  graph on $n\geq 2$ vertices. Then $\lambda _2 (G)=0$ if and only if $\qec(G)=0$.
\end{thm}
\newpage
\begin{rem}
	 By the above theorem, any graph $G$ with the property that either the second largest distance eigenvalue is zero or $\qec(G)=0$, satisfies $\lambda _2(G)=\qec(G)=0$. Unicyclic graphs with an even cycle, cacti with at least one even cycle, linear hexagonal chains, and a plethora of other graphs in \cite{aouchiche} have the second largest eigenvalue zero. Hence, for all those graph classes, $\qec(G)$ is the second largest eigenvalue of $D_G$.
\end{rem}

In 1977, Graham--Pollak \cite{Graham-Pollak} proved that $D_T$ for any tree $T$ is nonsingular and $\det(D_T)$ depends only on the number of vertices of $T$. Thus, a path graph with even number of vertices provides an example of a  graph with a nonsingular distance matrix and the quadratic embedding constant as a distance eigenvalue.  Our next main result gives a new subclass of  trees with the quadratic embedding constant as the second largest distance eigenvalue.

\begin{thm}\label{dstareig}
	Let $S_{n,n}$ be a double star graph (see Definition \ref{dstar}). Then \[\qec(S_{n,n})=\lambda _2(S_{n,n})=\frac{-(n+2)+\sqrt{(n+2)^2-8}}{2}.\] Moreover, in the definition of $\qec(S_{n,n})$, the maximum is attained at a vector $f\in \mathbb{R}^{2n}$ of the form $\frac{g}{\parallel g \parallel }$, where
	\begin{equation}\label{dstareq}
		g=\begin{pmatrix}
			\ \  y \\
			-y \\
		\end{pmatrix},~ y\in \mathbb{R}^n \hbox{~~~and~~~} y_i=\left\{\begin{array}{cl} \frac{n+\sqrt{(n+2)^2-8}}{2},&\mbox{if }i=1,\\-1,&\mbox{otherwise.}\end{array}\right.
	\end{equation}
\end{thm}

We, now turn our focus to investigating the quadratic embedding constant of $G$ using graph operations. It can be challenging to determine the quadratic embedding constant of a graph sometimes, but if we factorize the given $G$ into smaller graphs whose quadratic embedding constants are known, $\qec(G)$ can then be represented in terms of the quadratic embedding constant of the smaller graphs. Thus, a natural question arises: Can the quadratic embedding constant of $G=G_1\#G_2$ be computed using $\qec(G_1)$ and $\qec(G_2)$, where $\#$ denotes some graph operation? The quadratic embedding constant of $G$ using several graph operations, e.g., Cartesian product \cite{zakiyyah}, star product \cite{baskoro,mo}, lexicographic product and joining \cite{lou} of graphs studied in the literature. In this article, we study the cluster $G_1\{G_2\}$ of two simple, connected  graphs $G_1$ and $G_2$ (See Definition \ref{defncluster}). Our final result provides $\qec(G\{K_n\})$ and $\qec(G\{K_{1,{n-1}}\})$, in terms of the quadratic embedding constant of the base graph $G$.

\begin{thm}\label{clustercomplete}
	Let $n\geq 2$ be an integer and $G$ be an arbitrary graph. Then
	\begin{itemize}
		\item[(i)] For a complete graph $K_n$ on $n\geq 2$ vertices
		\[	\qec(G\{K_n\})=\frac{(n\qec(G)-n)+\sqrt{(n\qec(G)-n)^2+4n\qec(G)}}{2}.\]
		
		\item[(ii)] For a star graph $K_{1,n-1}$ on $n\geq 2$ vertices
		\[		\qec(G\{K_{1,n-1}\})=\frac{(n\qec(G)-2)+\sqrt{(n\qec(G)-2)^2+8\qec(G)}}{2}.\]
	\end{itemize}
	\end{thm}
As  applications of this result: (a) We generate new families of  examples of graphs of QE class. (b) We obtain the quadratic embedding constant of a simple corona graph.

\noindent\textbf{Organization of the paper:} The remaining sections are devoted to proving our main results above. In section \ref{sectree},  we recall the quadratic  embedding constants of a complete bipartite graph and of the star product of graphs, and prove Theorem \ref{ulbound}. In section \ref{specqec}, we prove Theorems \ref{sns} and \ref{dstareig}. We also provide a formula for the quadratic embedding constant of a general double star tree $S_{m,n}$ -- see Theorem \ref{gdoublestar}. In the final section, we prove Theorem \ref{clustercomplete}, followed by several applications.

\section{Sharp bound for the Quadratic Embedding Constant of Trees}\label{sectree}
We begin by proving Theorem \ref{ulbound}: for any tree $T_n$ on $n$ vertices, $\qec(T_n)$ lies between $\qec(P_n)$ and $\qec(K_{1,{n-1}})$. The proof requires a couple of preliminary results; the first gives the quadratic embedding constant of complete bipartite graphs.

\begin{thm}\label{qeccbip}\cite[Theorem 2.8]{zakiyyah}
	Let $K_{m,n}$ be a complete bipartite graph on $(m+n)$ vertices. Then
	\begin{center}
		$\qec(K_{m,n})=\frac{2(mn-m-n)}{m+n}$\ \ \  for all \ \ \  $m\geq 1$, $n\geq 1$.
	\end{center}
	In particular, for a star graph $K_{1,n-1}$ on $n$ vertices, $\qec(K_{1,n-1})=-\frac{2}{n}$.
\end{thm}
The next result concerns the quadratic embedding constant of  the star product of two graphs. Recall that the star product of $G_1=(V_1,E_1)$ and $G_2=(V_2,E_2)$ with respect to distinguished vertices $\sigma _1\in V_1$ and $\sigma _2\in V_2$ is also called their coalescence, is denoted by $G_1\star G_2$ and is obtained by joining $G_1$ and $G_2$ at the  vertices $\sigma _1$
and $\sigma _2$.

\begin{prop}\label{star}\cite[Proposition 4.3]{mo}
Let $G_1$ and $G_2$ be two graphs on $n_1+1$ and $n_2+1$ vertices with $Q_1=\qec(G_1)$ and $Q_2=\qec(G_2)$ respectively. 
 Let $M=M(n_1,n_2;-Q_1,-Q_2)$ be the conditional infimum of
	\[		\Phi (\alpha ,x^1,x^2)=\sum _{i=1}^2 (-Q_i)\{\langle x^i,x^i \rangle +\langle e,x^i \rangle ^2\}, \ \ \ \ \ \  \alpha \in \mathbb{R}, \ \  x^i\in \mathbb{R}^{n_i},
	\]
	subject to
	\[	
		\alpha ^2+\sum _{i=1}^2\langle x^i,x^i \rangle =1\ \ \   and\ \ \  \alpha +\sum _{i=1}^2\langle e,x^i \rangle =0.
	\]
	Then, $\qec(G_1\star G_2)\leq -M$.
\end{prop}

The third preliminary result establishes a relationship between the quadratic embedding constants of a graph and its subgraphs.
\begin{thm}\cite[Theorem 3.1]{zakiyyah}\label{iembedded}
	Let $G$ be a connected graph and $H$ be a connected subgraph of $G$. If $H$ is isometrically embedded in $G$, then $\qec(H)\leq \qec(G)$.
\end{thm}

With these preliminaries at hand, our first main result follows.
\begin{proof}[Proof of Theorem~\ref{ulbound}]
 First, we prove the left inequality -- $\qec(P_n) \leq \qec(T_n)$. Suppose $T_n$ is a tree other than the path graph $P_n$. Then $T_n$ has a vertex that is adjacent to at least three vertices, and 
 the graph induced by them is a subtree $K_{1,3}$ that is isometrically embedded in $T_n$. By Theorems \ref{qeccbip} and \ref{iembedded},
\begin{equation}\label{fi}
	-\frac{1}{2}=\qec(K_{1,3})\leq \qec(T_n).
\end{equation}
Also, from \cite[Section 4]{baskoro}, we have
\begin{equation}\label{si}
	\qec(P_2)<QEC(P_3)<\cdots <\qec(P_n)<\qec(P_{n+1})<\cdots \longrightarrow -\frac{1}{2}.
\end{equation}
From equation (\ref{fi}) and (\ref{si}), we have $\qec(P_n)< \qec(T_n)$.

Now, we prove the right inequality. Let $T_n$ be a tree with $n$ vertices distinct from the star graph $K_{1,{n-1}}$. We claim that $\qec(T_n)< \qec(K_{1,n-1})$. We prove this by induction on $n\geq 4$. In the base case, the only tree other than the star $K_{1,3}$ on 4 vertices is $P_4$ and $\qec(P_4)< -\frac{1}{2} =\qec(K_{1,3})$. 
For the induction step, let $T_{r+1}$ be a tree on $(r+1)$ vertices other than $K_{1,r}$, with $r\geq 4$. Then $T_{r+1}$ has a pendant vertex $v$ such that the tree $T_{r+1}-v=T_r$ is different from the star graph $K_{1,{r-1}}$. Note that $T_{r+1}=K_2\star T_r$. 
 By the induction hypothesis,
\begin{equation}\label{i1}
	\qec(T_r)< \qec(K_{1,r-1})=-\frac{2}{r}.
\end{equation}
Note that $\qec(K_2)=-1$. By Proposition \ref{star},
\begin{equation}\label{i2}
	\qec(T_{r+1})=\qec(K_2\star T_r) \leq -M(1,r-1;1,-\qec(T_r)).
\end{equation}
Since $-\qec(T_r)> \frac{2}{r}$, by \cite[Theorem 3.5 and Proposition 2.3]{mo},
\begin{equation}\label{i3}
	M(1,r-1;1,\frac{2}{r})< M(1,r-1;1,-\qec(T_r)).
\end{equation}
By \cite[Theorem 3.5]{mo}, $M(1,r-1;1,\frac{2}{r})$ is the minimal solution of
\[	\frac{1}{1\cdot 1+1-\lambda } + \frac{r-1}{\frac{2}{r}\cdot (r-1)+\frac{2}{r}-\lambda } = \frac{1}{\lambda},\]
and a straightforward calculation shows that $M(1,r-1;1,\frac{2}{r})=\lambda= \frac{2}{r+1}$.
Using (\ref{i2}) and (\ref{i3}), we get
\begin{center}
	$\qec(T_{r+1}) < -M(1,r-1;1,\frac{2}{r})=-\frac{2}{r+1}=\qec(K_{1,r})$.
\end{center}
This completes the induction step.
\end{proof}
\section{Quadratic Embedding Constants  and distance spectra of graphs}\label{specqec}

We next prove Theorem \ref{sns}. The proof employs the following preliminary result that computes the quadratic embedding constant using Lagrange multipliers.
\begin{prop}\cite[Proposition 4.1]{zakiyyah} \label{qecformula}
Let $G=(V,E)$ be a graph on $n\geq 3$ vertices. Identifying $C(V)$ with $\mathbb{R}^n$, let $\mathcal{S}(D_G)$ be the set of all stationary points $(f,\lambda ,\mu)\in \mathbb{R}^n\times \mathbb{R} \times \mathbb{R}$ of
\begin{center}
$\phi (f,\lambda ,\mu)=\langle f,D_Gf \rangle -\lambda (\langle f,f \rangle -1)-\mu \langle e,f \rangle ,$
\end{center}
or equivalently, $(f,\lambda ,\mu)\in \mathbb{R}^n\times \mathbb{R} \times \mathbb{R}$ satisfying  the system of the following three equations
\begin{center}
$(D_G-\lambda I)f=\frac{\mu}{2}\ e$,\ \ \ \  $\langle f,f \rangle =1$\ \  and\ \  $\langle e,f \rangle =0.$
\end{center}
Then $\mathcal{S}(D_G)$ is nonempty and
\begin{center}
$\qec(G) = \max \{\lambda : (f,\lambda ,\mu)\in \mathcal{S}(D_G) \}$.
\end{center}
\end{prop}
\begin{proof}[Proof of Theorem~\ref{sns}]
Suppose that $\qec(G)=0$. Then $\lambda _2(G)\leq 0$ by \eqref{minmax}. We claim that $\lambda _2(G) =0$. Indeed, suppose $\lambda _2(G) \neq 0$. Then $D_G$ is a nonsingular matrix by \eqref{minmax}. By Proposition \ref{qecformula}, there exist $f \in \mathbb{R}^n$ and $\mu \in \mathbb{R}$ such that
\begin{center}
$(D_G-0\cdot I)f=\frac{\mu}{2}\ e$,\ \ \ \  $\langle f,f \rangle =1$\ \  and\ \  $\langle e,f \rangle =0$.
\end{center}
Since $D_G$ is nonsingular, $f=\frac{\mu}{2}D_G^{-1}e$ and $\frac{\mu ^2}{4}\langle D_G^{-1}e,D_G^{-1}e \rangle =1$. Thus
\begin{equation}\label{zero}
	\mu \neq 0 ~~~\mbox{and}~~~ \langle e,D_G^{-1}e \rangle =0.
\end{equation}
By \cite[Lemma 8.3]{bapat} and using \eqref{zero}, we have $\det(D_G+J)=\det(D_G)$. Let $\beta_1\geq \cdots \geq \beta_n$ be the eigenvalues of $D_G+J$. Since $D_G$ is symmetric, by \cite[Corollary 4.3.12]{horn}, $\lambda _i (G)\leq \beta _i$ for all $1\leq i \leq n$. 
  Let $x>0$ be the Perron vector of $D_G$. By the Perron--Frobenius Theorem, any eigenvector corresponding to $\lambda _1(G)$ is a constant multiple of $x$. Let $y$ be an arbitrary eigenvector corresponding to $\lambda _1(G)$. Then
  \[D_{G} y=\lambda _1(G) y  ~~~~~~\mbox{~~~~~and~~~~~}~~~~~ \sum _{i=1}^ny_i\neq 0.\]
Since  $Jy\neq 0$, again by \cite[Corollary 4.3.12]{horn}, $\lambda _1(G)< \beta _1$. Thus \[\det(D_G)=\prod _{i=1}^n\lambda_i (G)\neq \prod _{i=1}^n\beta_i=\det(D_G+J),\] a contradiction. Hence $\lambda _2(G)$ must be zero.

To prove the converse, let $\lambda _2(G)=0$. Then $ \qec(G)\geq 0$. If $\qec(G)=0$, then we are done. Suppose that $\qec(G)>0$.  Since $0=\lambda _2(G)<\qec(G)<\lambda _1(G)$, the matrix $D_G-\qec(G) I$ is invertible. Again, by Proposition \ref{qecformula}, there exist  $f \in \mathbb{R}^n$ and $\mu \in \mathbb{R}$ 
 such that
\begin{center}
$(D_G-\qec(G) I)f=\frac{\mu}{2}\ e$,\ \ \ \  $\langle f,f \rangle =1$\ \  and\ \  $\langle e,f \rangle =0,$
\end{center}
and so $\det(D_G-\qec(G) I+J)=\det(D_G-\qec(G) I)$. The same argument can be adapted, as in the first half of the proof, to show that $\lambda _i (G) \leq \beta _i $ for all $2\leq i\leq n$ and $\lambda _1(G)<\beta _1 $. Hence
 \[\det(D_G-\qec(G) I)=\prod _{i=1}^n(\lambda_i(G)-\qec(G) )\neq \prod _{i=1}^n(\beta_i-\qec(G) )=\det(D_G-\qec(G) I+J),\]  a contradiction. Thus $\qec(G)=0$.
\end{proof}
As a consequence, we characterize all singular graphs of QE class.
\begin{cor}
Let $G$ be a graph of QE class. Then $D_G$ is singular if and only if $\qec(G)=0$.
\end{cor}
\begin{proof}
	Suppose  that $G$ is a graph of QE class. By Schoenberg’s Theorem \ref{schointro}, $\qec(G)\leq 0$. Since  $\lambda _2 (G)\leq \qec(G)<\lambda _1(G)$, $\det(D_G)=0$ if and only if $\lambda _2(G) = 0$. Hence the result follows by Theorem \ref{sns}.
\end{proof}
Next, we will look at graphs with nonsingular distance matrices. In general, the quadratic embedding constant is not a distance eigenvalue for every graph with a nonsingular distance matrix, e.g.,  complete bipartite graphs $K_{m,n}$ ($m \neq n$), path graphs with an odd number of vertices, star graphs, etc. In Theorem \ref{dstareig}, we identify a subclass of trees,  the double star graphs for which the QEC is an eigenvalue of the distance matrix. To prove Theorem \ref{dstareig}, we first recall the definition of the double star graph.
\begin{defn}\label{dstar}
	A double star graph is a tree denoted by $S_{m,n}$, obtained by adding an edge between the center vertices of the star graphs $K_{1,m-1}$ and  $K_{1,n-1}$.
\end{defn}
\begin{ex}
	Example of a double star graph on $8$ vertices.
	\begin{figure}[H]
		\begin{center}
			\begin{tikzpicture}[scale=1]
				\draw  (4.,3.)-- (3.,4.);
				\draw  (4.,3.)-- (3.,2.);
				\draw  (4.,3.)-- (6.,3.);
				\draw  (6.,3.)-- (6.,4.);
				\draw  (6.,3.)-- (6.,2.);
				\draw  (6.,3.)-- (7.,2.);
				\draw  (6.,3.)-- (7.,4.);
				\begin{scriptsize}
					\fill (4.,3.) circle (2.5pt);
					\draw (4.14,3.32) node {$1$};
					\fill (6.,3.) circle (2.5pt);
					\draw (5.7,3.32) node {$4$};
					\fill (3.,4.) circle (2.5pt);
					\draw (3.,4.4) node {$2$};
					\fill (3.,2.) circle (2.5pt);
					\draw (3.,1.6) node {$3$};
					\fill (7.,4.) circle (2.5pt);
					\draw (7.,4.4) node {$6$};
					\fill (7.,2.) circle (2.5pt);
					\draw (7.,1.65) node {$7$};
					\fill (6.,4.) circle (2.5pt);
					\draw (6.,4.4) node {$5$};
					\fill (6.,2.) circle (2.5pt);
					\draw (6.,1.65) node {$8$};
				\end{scriptsize}
			\end{tikzpicture}
		\end{center}
		\caption{$S_{3,5}$}
	\end{figure}
\end{ex}
\begin{proof}[Proof of Theorem~\ref{dstareig}]
	Note that $S_{n,n}=K_2\{K_{1,n-1}\}$ (see Definition \ref{defncluster}). By Theorem \ref{clustercomplete},
	\[\qec(S_{n,n})=\frac{-(n+2)+\sqrt{(n+2)^2-8}}{2}.\]
	Next, we show that $\qec(S_{n,n})$ is an eigenvalue of $D_{S_{n,n}}$. Let $\{1, 2, \ldots n, n+1, \ldots ,2n\}$ be the vertex set of $S_{n,n}$ such that one $K_{1,n-1}$ is induced by $\{1,2,\ldots ,n\}$ with the vertex 1 as its centre and the other $K_{1,n-1}$ is induced by vertices $\{n+1,n+2,\ldots ,2n\}$ with the centre vertex $n+1$. Then  $D_{S_{n,n}}=\begin{pmatrix}
		P & Q \\
		Q & P\\
	\end{pmatrix}$, where $P=(p_{ij}),Q=(q_{ij})\in \mathbb{R}^{n\times n}$ are two symmetric matrices of the form
	\begin{center}
		$p_{ij}=\left\{\begin{array}{cl} 0,&\mbox{if }i=j,\\1,&\mbox{if }i=1\mbox{ and }2\leq j\leq  n,\\2,&\mbox{otherwise.}\end{array}\right.$\ \ \ \ 
		and\ \ \ \ 
		$q_{ij}=\left\{\begin{array}{cl} 1,&\mbox{if }i=j=1,\\2,&\mbox{if }i=1\mbox{ and }2\leq j\leq  n,\\3,&\mbox{otherwise.}\end{array}\right.$
	\end{center}
	Note that
	\[q_{ij}-p_{ij}=\left\{\begin{array}{cl} 3,&\mbox{if }i=j\mbox{ and }2\leq i\leq  n,\\1,&\mbox{otherwise.}\end{array}\right.\]
	Define the vectors $g$ and $y$ as in \eqref{dstareq}. Then \[(Q-P)y=\frac{-(n+2)+\sqrt{(n+2)^2-8}}{2}\Big (-\frac{n+\sqrt{(n+2)^2-8}}{2},1,1,\ldots ,1 \Big )^T\] and 
	\[D_{S_{n,n}}g=\begin{pmatrix}
		(P-Q)y \\
		(Q-P)y \\
	\end{pmatrix}=\frac{-(n+2)+\sqrt{(n+2)^2-8}}{2}
	\begin{pmatrix}
		\ \  y \\
		-y \\
	\end{pmatrix}=\qec(S_{n,n})\ g.\]
Hence $\qec(S_{n,n})$ is an eigenvalue of $D_{S_{n,n}}$. Since $\lambda _2(S_{n,n})\leq \qec(S_{n,n})< \lambda _1(S_{n,n})$, we have
\[\lambda _2(S_{n,n})=\qec(S_{n,n}) \mbox{~~~~and~~~~}\left\langle \frac{g}{\parallel g \parallel},D_{S_{n,n}}\frac{g}{\parallel g \parallel} \right\rangle =\qec(S_{n,n}).\qedhere\]
\end{proof}

In the last part of this section, we provide the quadratic embedding constant of a general double star graph $S_{m,n}$. Note that $\qec(S_{m,n})$ is the maximum of $\langle f,D_{S_{m,n}} f \rangle$ over the uncountable set of all vectors $f$ in the unit sphere orthogonal to $e$. Our next result avoid this rout  and compute $\qec(S_{m,n})$ via  the largest root of a polynomial of degree 3. The proof requires a basic result that gives us a lower bound for the quadratic embedding constant of any connected graph.
\begin{prop}\label{qeccomplete}\cite{lou}
	Let $G$ be a graph on $n$ vertices. Then, $\qec(G)\geq -1$. Moreover, equality holds if and only if $G$ is a complete graph $K_n$.
\end{prop}
From the above result and Proposition \ref{qecformula}, to calculate the quadratic embedding constant of a graph that is not complete, it is sufficient to take the maximum over $\lambda >-1$ such that $(f,\lambda,\mu)$ is a stationary point of $\phi (f,\lambda,\mu )$.

\begin{thm}\label{gdoublestar}
	Let $m,n\geq 1$ be integers. Then $\qec(S_{m+1,n+1})=-2+2t$, where $t$ is the largest root of the polynomial \[(m+n+2)t^3-(m+n+1-mn)t^2-2mnt +mn=0.\]
\end{thm}
\begin{proof}
Let $\{1, 2, \ldots , m+1, m+2, \ldots , m+n+2\}$ be the vertex set of $S_{m+1,n+1}$ such that $K_{1,m}$ is induced by $\{1,2,\ldots ,m,m+n+1\}$ with the vertex $m+n+1$ as its centre and $K_{1,n}$ is induced by vertices $\{m+1,m+2,\ldots ,m+n,m+n+2\}$ with the centre vertex $m+n+2$. Then, the distance matrix of $S_{m+1,n+1}$ is an $(m+n+2)\times (m+n+2)$ matrix of the form 
\begin{center}
	$\begin{pmatrix}
		A & P & Q \\
		P^T & B & R \\
		Q^T & R^T & C \\
	\end{pmatrix}$,
\end{center}
where $A=2(J-I)\in \mathbb{R}^{m\times m}$, $B=2(J-I)\in \mathbb{R}^{n\times n}$, $C=(J-I)\in \mathbb{R}^{2\times 2}$, $P=3J\in \mathbb{R}^{m\times n}$,
\begin{center}
	$Q=\begin{pmatrix}
		1 & 2 \\
		1 & 2 \\
		\vdots & \vdots \\
		1 & 2 \\
	\end{pmatrix}\in \mathbb{R}^{m\times 2}$\ \ \ \  and \ \ \ \ 
	$R=\begin{pmatrix}
		2 & 1 \\
		2 & 1 \\
		\vdots & \vdots \\
		2 & 1 \\
	\end{pmatrix}\in \mathbb{R}^{n\times 2}$.
\end{center}
Let $f=(x^T,y^T,z^T)^T\in \mathbb{R}^{(m+n+2)}$, where $x\in \mathbb{R}^{m}$, $y\in \mathbb{R}^{n}$ and $z\in \mathbb{R}^{2}$ such that $\langle e,f\rangle =0$ and $\langle f,f\rangle =1$. Then a straightforward but longwinded calculation shows that
\begin{align*}
	\langle f,D_{S_{m+1,n+1}}f\rangle &= x^TAx+y^TBy+z^TCz+2x^TPy+2x^TQz+2y^TRz \\
	&= 2\Big \{\Big (\sum _{i=1}^m x_i\Big )^2-\sum _{i=1}^m x_i^2\Big \}+2\Big \{\Big (\sum _{j=1}^n y_j\Big )^2-\sum _{j=1}^n y_j^2\Big \}+\Big (\sum _{k=1}^2 z_k\Big )^2-\sum _{k=1}^2 z_k^2\\
	& +6\Big (\sum _{i=1}^m x_i\Big )\Big (\sum _{j=1}^n y_j\Big )+2(z_1+2z_2)\Big (\sum _{i=1}^m x_i\Big )+2(2z_1+z_2)\Big (\sum _{j=1}^n y_j\Big )\\
	&=-2+2\Big \{z_2^2-\Big (\sum _{i=1}^m x_i\Big )^2-2z_1\Big (\sum _{i=1}^m x_i\Big )\Big \}.
\end{align*}
Let $\Psi (x,y,z)=\Psi (x_1,\ldots ,x_m,y_1,\ldots , y_n,z_1,z_2):=z_2^2-\Big (\sum _{i=1}^m x_i\Big )^2-2z_1\Big (\sum _{i=1}^m x_i\Big )$. Then there exists a symmetric matrix $M$ such that $\Psi (x,y,z)=\langle f,Mf \rangle$. Thus
\begin{align*}
	\qec(S_{m+1,n+1}) &= \max\{\langle f,D_{S_{m+1,n+1}}f \rangle ;\  f\in \mathbb{R}^{(m+n+2)\times (m+n+2)},\  \langle f,f \rangle = 1,\  \langle e,f \rangle = 0 \} \\
	&=-2+2\ \max \{\langle f,Mf \rangle;\  f= (x,y,z)^T,\  \langle f,f \rangle = 1,\  \langle e,f \rangle = 0\}.
\end{align*}
We use the Lagrange multiplier method to solve the above maximization problem. Set
\[F(f, \lambda  ,\mu ):=\langle f,Mf \rangle -\lambda  (\langle f,f \rangle - 1)-\mu  \langle e,f \rangle.\] Then all the stationary points of $F(f, \lambda ,\mu )$ satisfy $(M-\lambda I)f=\frac{\mu }{2}$, $\langle f,f \rangle = 1$ and $\langle e,f \rangle =0$. Since $K_{1,{m+1}}$ is isometrically embedded in $S_{{m+1},{n+1}}$, it suffices to consider $\lambda \in \mathbb{R}$ with $\frac{m+1}{m+2}\leq \lambda<1$. By solving $(M-\lambda I)f=\frac{\mu }{2}$, we get
\begin{align*}
	x_i &=\frac{\mu }{2}\cdot \frac{\lambda -1}{m-\lambda (m+\lambda)},\ \   1\leq i \leq m, \\
	y_j &=-\frac{\mu }{2\lambda },\ \   1\leq j \leq n,\\
	z_1 &=\frac{\mu }{2}\cdot \frac{\lambda }{m-\lambda (m+\lambda )},\\
	z_2 &=\frac{\mu }{2-2\lambda }.
\end{align*}
Using $\langle f,f \rangle = 1$ and $\langle e,f \rangle =0$, we have $p(\lambda )=(m+n+2)\lambda ^3-(m+n+1-mn)\lambda ^2-2mn\lambda  +mn=0$. So,
\begin{align*}
	\qec(S_{m+1,n+1}) &= -2+2\ \max \{\langle f,Mf \rangle;\  f= (x,y,z)^T,\  \langle f,f \rangle = 1,\  \langle e,f \rangle = 0\} \\
	&= -2+2\ \max \{\lambda ;\  (M-\lambda I)f=\frac{\mu }{2},\  \langle f,f \rangle = 1,\  \langle e,f \rangle = 0\}\\
	&=-2+2\ \max \{\lambda ;\  p(\lambda )=0\}.
\end{align*}
Next, we claim that the roots of $p(\lambda)=0$ are real. Since $p(0)=mn>0$, $p(1)=1>0$, and $p(\frac{m+n}{m+n+2})=-\Big (\frac{m-n}{m+n+2} \Big )^2\leq 0$, $p(\lambda)$ has at least one positive root. By Descartes' rule of sign,  $p(\lambda)$ has two positive and one negative root. This completes the proof.
\end{proof}

\section{Quadratic Embedding Constant of Cluster of Graphs} \label{clustsec}
In this final section, we prove theorem \ref{clustercomplete} -- deriving the quadratic embedding constants of the clusters of an arbitrary graph with  $K_n$ and $K_{1,n-1}$. First we recall the definition of cluster of two graphs in general.
\begin{defn}\label{defncluster}\cite{zhang}
Let $G_1$ and $G_2$ be two graphs and $V(G_1)$ be the vertex set of $G_1$. The cluster $G_1\{G_2\}$ is obtained by taking one copy of $G_1$ and $\vert V(G_1) \vert$ copies of a rooted graph $G_2$, and by identifying the root {\rm (}should be same for each copy of $G_2${\rm )} of the i-th copy of $G_2$ with the i-th vertex of $G_1$, $i=1,2,\ldots ,\vert V(G_1) \vert$.
\end{defn}
\begin{ex}
Example of a cluster graph with $G_2=K_3$:
\begin{figure}[H]
\begin{center}
\begin{tikzpicture}[scale=1]
\draw  (1.,3.)-- (1.,1.);
\draw  (1.,1.)-- (3.,1.);
\draw  (3.,1.)-- (3.,3.);
\draw  (1.,3.)-- (3.,1.);
\begin{scriptsize}
\fill (1.,3.) circle (2.5pt);
\draw (2.,0.25) node {$G_1$};
\fill (1.,1.) circle (2.5pt);
\fill (3.,3.) circle (2.5pt);
\fill (3.,1.) circle (2.5pt);
\end{scriptsize}
\end{tikzpicture}
\hspace{1cm}
\begin{tikzpicture}[scale=1]
\draw  (3.,3.)-- (4.,4.);
\draw  (4.,4.)-- (5.,3.);
\draw  (5.,3.)-- (3.,3.);
\begin{scriptsize}
\fill (3.,3.) circle (2.5pt);
\draw (4.,2.25) node {$G_2$};
\fill (4.,4.) circle (2.5pt);
\draw (5.,3.) circle (4.5pt);
\fill (5.,3.) circle (2.5pt);
\end{scriptsize}
\end{tikzpicture}
\hspace{2cm}
\begin{tikzpicture}[scale=1]
\draw  (2.,3.)-- (4.,3.);
\draw  (4.,3.)-- (3.,4.);
\draw  (3.,4.)-- (2.,3.);
\draw  (2.,1.)-- (3.,2.);
\draw  (3.,2.)-- (4.,1.);
\draw  (4.,1.)-- (2.,1.);
\draw  (6.,1.)-- (7.,2.);
\draw  (7.,2.)-- (8.,1.);
\draw  (8.,1.)-- (6.,1.);
\draw  (6.,3.)-- (7.,4.);
\draw  (7.,4.)-- (8.,3.);
\draw  (8.,3.)-- (6.,3.);
\draw  (4.,3.)-- (4.,1.);
\draw  (4.,1.)-- (6.,1.);
\draw  (6.,1.)-- (4.,3.);
\draw  (6.,3.)-- (6.,1.);
\begin{scriptsize}
\fill (4.,3.) circle (2.5pt);
\draw (4.35,3.) node {$1$};
\fill (6.,3.) circle (2.5pt);
\draw (5.7,3.) node {$4$};
\fill (4.,1.) circle (2.5pt);
\draw (4.,0.65) node {$2$};
\fill (6.,1.) circle (2.5pt);
\draw (6.,0.65) node {$3$};
\fill (7.,4.) circle (2.5pt);
\draw (7.,4.37) node {$8$};
\fill (8.,3.) circle (2.5pt);
\draw (8.36,3.) node {$12$};
\fill (7.,2.) circle (2.5pt);
\draw (7.,2.37) node {$7$};
\fill (8.,1.) circle (2.5pt);
\draw (8.35,0.65) node {$11$};
\fill (3.,2.) circle (2.5pt);
\draw (3.,2.37) node {$6$};
\fill (2.,1.) circle (2.5pt);
\draw (1.65,0.65) node {$10$};
\fill (2.,3.) circle (2.5pt);
\draw (1.67,3.) node {$9$};
\fill (3.,4.) circle (2.5pt);
\draw (3.,4.37) node {$5$};
\draw (5.,0.25) node {$G_1\{G_2\}$};
\end{scriptsize}
\end{tikzpicture}
\end{center}
\caption{}
\label{figure1}
\end{figure}
\end{ex}

Cluster graphs have important applications in Chemistry -- all composite molecules consisting of some amalgamation over a central submolecule can be understood as generalized cluster graphs. For instance, they can be used to understand some issues in metal-metal interaction in some molecules, since a cluster graph structure can be easily found. In \cite{zhang}, the Kirchhoff index formulae for cluster graphs are presented in terms of the pieces.

To prove Theorem \ref{clustercomplete}, we need a basic result. The following lemma provides two examples of monotonically increasing functions that will be crucial in proving Theorem \ref{clustercomplete}. We only prove the monotonicity of the first function, and that of the second function may be shown similarly.
	\begin{lem}\label{realrelation}
		Let $Y,Z:\mathbb{R}\rightarrow \mathbb{R}$ be two real-valued functions defined by
		\begin{center}
			\rm{(i)} $Y(b):=\frac{(nb-n)+\sqrt{(nb-n)^2+4nb}}{2}$ \ \ \ \  and  \ \ \ \  \rm{(ii)} $Z(b):=\frac{(nb-2)+\sqrt{(nb-2)^2+8b}}{2},$
		\end{center}
		where $n\geq 2$ is a natural number. Then $Y$ and $Z$ are monotonically increasing functions.
\end{lem}
\begin{proof}
	(i) First, note that the term inside the square root in $Y(b)$ is always positive. For $b\geq 0$, it is obvious, while for $b<0$, 
	\begin{center}
		$(nb-n)^2+4nb = (nb)^2+n^2+2nb(2-n)> 0$.
	\end{center}
	Let $F_b:=\sqrt{(b-1)^2+\frac{4b}{n}}$, where $b\in \mathbb{R}$. Then $F_b>0$ and
	\begin{equation*}
		(b-1)^2+\frac{4b}{n}=\Big (b-1+\frac{2}{n}\Big )^2+(\frac{4}{n}-\frac{4}{n^2}) > \Big (b-1+\frac{2}{n}\Big )^2 = \Big (1-b-\frac{2}{n}\Big )^2,
	\end{equation*}
	which implies $F_b>(1-b-\frac{2}{n})$. Now, for $b\geq c$,
	\begin{align*}
		Y(b)-Y(c) &= \frac{(nb-n)+\sqrt{(nb-n)^2+4nb}}{2} - \frac{(nc-n)+\sqrt{(nc-n)^2+4nc}}{2}\\
		&=\frac{n}{2}\Big \{(b-c)+\frac{(b-1)^2+\frac{4b}{n}-(c-1)^2-\frac{4c}{n}}{\sqrt{(b-1)^2+\frac{4b}{n}}+\sqrt{(c-1)^2+\frac{4c}{n}}} \Big \}\\
		&=\frac{n}{2} \Big (\frac{b-c}{F_b+F_c}\Big ) ( F_b+F_c+b+c-2+\frac{4}{n}) \geq 0.
	\end{align*}
This concludes our proof.
\end{proof}
With Lemma \ref{realrelation} and Proposition \ref{qeccomplete} in hand, we now prove Theorem \ref{clustercomplete}.
\begin{proof}[Proof of Theorem~\ref{clustercomplete}]
(i) For the sake of notional convenience, set $H:=G\{K_n\}$. Let the vertex set of $G$ be $V=\{1,2,\ldots , m\}$ and the vertex set of the $i$-th copy of $K_n$ be $V_i=\{i,i+m,i+2m\ldots , i+(n-1)m\}$, $1\leq i\leq m$ (see the labeling of vertices in Figure \ref{figure1}). If $\Tilde{V}$ is the vertex set of $H$, then $\Tilde{V}=\cup _{i=1}^mV_i$. Then, $D_H$ can be written as
\[\begin{pmatrix}
D_G & D_G+J & D_G+J & \hdots & D_G+J \\
D_G+J & D_G+2(J-I) & D_G+2J-I & \hdots & D_G+2J-I \\
D_G+J & D_G+2J-I & D_G+2(J-I) & \hdots & D_G+2J-I \\
\vdots & \vdots & \vdots & \ddots & \vdots \\
D_G+J & D_G+2J-I & D_G+2J-I &\hdots  & D_G+2(J-I)\\
\end{pmatrix}.\]
Let $S(D_H)$ be the set of all $(f,\lambda ,\mu)\in (C(\Tilde{V})\cong \mathbb{R}^{mn})\times \mathbb{R} \times \mathbb{R}$ satisfying
\begin{align}
(D_H-\lambda I)f &=\frac{\mu}{2}\  e \label{mainequation},\\
\langle f,f \rangle &=1 \label{one},\\
\langle e,f \rangle &=0. \label{notone}
\end{align}
By Proposition \ref{qecformula}, $\qec(H)=\max \{\lambda :(f,\lambda ,\mu)\in S(D_H)\}$. Suppose $f=({x^1}^T,{x^2}^T,\ldots ,{x^n}^T)^T$, where $x^i\in \mathbb{R}^m$ for all $1\leq i\leq n$. Using \eqref{notone}, we have $\sum _{i=1}^n\langle e,x^i \rangle =0$ which in turn gives
\begin{equation}\label{jsum}
J(\sum _{i=1}^nx^i)=\sum _{i=1}^n\langle e,x^i \rangle e=0.
\end{equation}
From (\ref{mainequation}), we obtain a system of $n$ equations. Using (\ref{jsum}) and subtracting the $i$-th equation from the first equation for $2\leq i\leq n$, we get the following system of $(n-1)$ equations:
\begin{align*}
(\lambda +2)x^2+x^3+\cdots +x^n=\lambda x^1,\\
x^2+(\lambda +2)x^3+\cdots +x^n=\lambda x^1,\\
\hdots \hdots \hdots \hdots \hdots \hdots \hdots \hdots \hdots \hdots\\
x^2+x^3+\cdots +(\lambda +2)x^n=\lambda x^1.
\end{align*}
By solving the above system of linear equations, we get
\begin{equation}\label{jxzero}
x^2=x^3=\cdots =x^n=\frac{\lambda}{\lambda +n}x^1.
\end{equation}
From (\ref{jsum}) and (\ref{jxzero}), $\frac{n\lambda +n}{\lambda +n}Jx^1=0$ which implies $Jx^1=0$, because $\lambda >-1$ (by the discussion immediately following Proposition \ref{qeccomplete}). Using the aforementioned observations, we can rewrite equations [\ref{mainequation}--\ref{notone}] as follows:
\begin{eqnarray}
	\Big \{\Big (1+(n-1)\frac{\lambda}{\lambda +n} \Big )D_G-\lambda I \Big \}x^1&=&\frac{\mu}{2}\  e \label{secondmain}\\
	\langle x^1,x^1 \rangle &=&\frac{1}{1+(n-1)\Big (\frac{\lambda}{\lambda +n} \Big )^2 }\label{secondone}\\
	\langle e,x^1 \rangle& =&0. \label{secondnotone}
\end{eqnarray}
On the other hand, to calculate $\qec(G)$, consider the set $S(D_G)$ of all the stationary points $(x,a,\gamma)\in (C(V)\cong \mathbb{R}^m)\times \mathbb{R}\times \mathbb{R}$ satisfying
\begin{equation}\label{normalmain}
(D_G-aI)x=\frac{\eta}{2}\ e,\ \ \ \  \langle x,x \rangle =1,\ \ \ \  \langle e,x \rangle =0.
\end{equation}
By looking at equations [\ref{secondmain}--\ref{normalmain}], the following relations
\begin{align*}
\lambda = \frac{(na-n)+\sqrt{(na-n)^2+4na}}{2}\\
x^1= \frac{1}{\sqrt{1+(n-1)\Big (\frac{\lambda}{\lambda +n} \Big )^2 }}\  x=\frac{\lambda +n}{\sqrt{(\lambda +n)^2+(n-1)\lambda ^2}}\ x,\\
\mu =\frac{1+(n-1)\Big (\frac{\lambda}{\lambda +n} \Big )}{\sqrt{1+(n-1)\Big (\frac{\lambda}{\lambda +n} \Big )^2 }}\  \eta =\frac{n(\lambda +1)}{\sqrt{(\lambda +n)^2+(n-1)\lambda ^2}}\  \eta ,
\end{align*}
give us a one to one correspondence between $S(D_G)$ and $S(D_H)$. So,
\begin{align*}
\qec(H) & = \max \{\lambda : (({x^1}^T,{x^2}^T,\ldots ,{x^n}^T)^T,\lambda ,\mu)\in S(D_H)\} \\
 & = \max \Big \{\frac{(na-n)+\sqrt{(na-n)^2+4na}}{2} : (x,a ,\eta)\in S(D_G)\Big \}.
\end{align*}
Now $a\leq \qec (G)$, and in particular for $\qec (G)$, there exist $\Tilde{x}\in \mathbb{R}^m$ and $\Tilde{\eta}\in \mathbb{R}$ such that $(\Tilde{x} ,\qec (G),\Tilde{\eta})\in S(D_G)$. Thus
\begin{equation}\label{lrelation}
\frac{(n\qec (G)-n)+\sqrt{(n\qec (G)-n)^2+4n\qec (G)}}{2} \leq \qec(H).
\end{equation}
Also, by Lemma \ref{realrelation},  $a\leq Q$ implies
\begin{equation}\label{rrelation}
\qec(H)\leq \frac{(n\qec (G)-n)+\sqrt{(n\qec (G)-n)^2+4n\qec (G)}}{2}.
\end{equation}
Hence, the conclusion follows from (\ref{lrelation}) and (\ref{rrelation}).

	(ii) Set $L:=G\{K_{1,n-1}\}$. Label the vertices of $L$ as in the preceding part of this proof. Then
\begin{center}
	$D_L=\begin{pmatrix}
		D_G & D_G+J & D_G+J & \hdots & D_G+J \\
		D_G+J & D_G+2(J-I) & D_G+2J & \hdots & D_G+2J \\
		D_G+J & D_G+2J & D_G+2(J-I) & \hdots & D_G+2J \\
		\vdots & \vdots & \vdots & \ddots & \vdots \\
		D_G+J & D_G+2J & D_G+2J &\hdots  & D_G+2(J-I)\\
	\end{pmatrix}.$
\end{center}
Let $S(D_L)$ be the set of all stationary points $(f,\lambda ,\mu)\in \mathbb{R}^{mn}\times \mathbb{R} \times \mathbb{R}$ satisfying
\begin{equation}\label{mainstar}
	(D_L-\lambda I)f=\frac{\mu}{2}\ e,\ \ \ \  \langle f,f \rangle =1,\ \ \ \  \langle e,f \rangle =0,
\end{equation}
and let $S(D_G)$ be the set of all stationary points $(x,a ,\eta)\in \mathbb{R}^m\times \mathbb{R} \times \mathbb{R}$ satisfying
\begin{equation}
	(D_G-aI)x=\frac{\eta}{2}\ e,\ \ \ \  \langle x,x \rangle =1,\ \ \ \  \langle e,x \rangle =0.
\end{equation}
Here, note that $K_{1,n}$ is isometrically embedded in $L$. So, $-\frac{2}{n+1}\leq \lambda$. Thus, it suffices to consider all stationary points in $S(D_L)$ with $-\frac{2}{n}<\lambda $. 
By a similar argument as in the first part of the proof, we can show that \eqref{mainstar} reduces to the following equations:
\begin{eqnarray}
	\Big \{\Big (1+(n-1)\frac{\lambda}{\lambda +2} \Big )D_G-\lambda I \Big \}x^1&=&\frac{\mu}{2}\  e \nonumber\\
	\langle x^1,x^1 \rangle &=&\frac{1}{1+(n-1)\Big (\frac{\lambda}{\lambda +2} \Big )^2 }\nonumber\\
	\langle e,x^1 \rangle &=&0,\nonumber
\end{eqnarray}
where $f=({x^1}^T,{x^2}^T,\ldots ,{x^n}^T)^T$, $x^i\in \mathbb{R}^m$.
\begin{align*}
	\qec(L) & = \max \{\lambda : (({x^1}^T,{x^2}^T,\ldots ,{x^n}^T)^T,\lambda ,\mu)\in S(D_L)\} \\
	& = \max \Big \{\frac{(na-2)+\sqrt{(na-2)^2+8a}}{2} : (x,a ,\eta)\in S(D_G)\Big \}.
\end{align*}
The remainder of the proof is  same as the last part of the above proof, except for the use of the second part of Lemma \ref{realrelation}.
\end{proof}
{\begin{rem}
		Since $\qec(K_n)=-1$ and $\qec(K_{1,n-1})=-\frac{2}{n}$, using Theorem \ref{clustercomplete}, one can verify that $\qec (K_{1,n-1}\{K_m\})=\qec (K_m\{K_{1,n-1}\})$ if and only if $m=n$.
\end{rem}
We conclude by discussing  a couple of applications of Theorem \ref{clustercomplete}. Our first application provides novel families of examples  of graphs of QE class.
\begin{cor}\label{qeclass}
	If $G$ is a graph of QE class, then the graphs $G\{K_n\}$ and $G\{K_{1,n-1}\}$ are of QE class.
\end{cor}
\begin{proof}
	If $G$ is a graph of QE class then $\qec(G) \leq 0$. By Lemma \ref{realrelation}, 
	\begin{align*}
		\frac{(n\qec(G)-n)+\sqrt{(n\qec(G)-n)^2+4n\qec(G)}}{2} & \leq  \frac{-n+\sqrt{n^2}}{2} \\
		\Rightarrow \qec(G\{K_n\}) & \leq 0;
	\end{align*}
	and
	\begin{align*}
		\frac{(n\qec(G)-2)+\sqrt{(n\qec(G)-2)^2+8\qec(G)}}{2} & \leq  \frac{-2+2}{2} \\
		\Rightarrow \qec(G\{K_{1,n-1}\}) & \leq 0.
	\end{align*}
	Thus $G\{K_n\}$ and $G\{K_{1,n-1}\}$ are of QE class.
\end{proof}
The next application provides the quadratic embedding constant of a simple corona graph.

\begin{defn}
Let $G_1$ and $G_2$ be two graphs defined on disjoint sets of $m$ and $n$ vertices, respectively. The corona $G_1\circ G_2$ of $G_1$ and $G_2$ is defined as the graph obtained by taking one copy of $G_1$ and $m$ copies of $G_2$, and then joining the $i$th vertex of $G_1$ to every vertex in the $i$th copy of $G_2$. If $G_2=K_1$, then the corona $G\circ K_1$ is called a simple corona graph.
\end{defn}
\begin{cor}\label{cgraph}
Let $G\circ K_1$ be a simple corona graph. Then,
\begin{equation*}
\qec(G\circ K_1)=\qec(G)-1+\sqrt{1+(\qec(G))^2}.
\end{equation*}
\end{cor}
Observe that for any tree $T$, $T\{K_{1,n-1}\}$ is again a tree. As a final application of Theorem \ref{clustercomplete}, we give a better lower and upper bound for this class of trees than Theorem \ref{ulbound}.
\begin{cor}
	Let $T$ be a tree on $m$ vertices. Then
	\[ \qec(P_{mn})<\qec(P_m\{K_{1,n-1}\}) \leq \qec(T\{K_{1,n-1}\}) \leq \qec(K_{1,m-1}\{K_{1,n-1}\})< \qec(K_{1,mn-1}).\]
\end{cor}
\begin{proof}
	Note that the number of vertices in $T\{K_{1,n-1}\}$ is $mn$. By Theorems \ref{ulbound}, \ref{clustercomplete} and Lemma \ref{realrelation},
	\[\qec(P_m\{K_{1,n-1}\}) \leq \qec(T\{K_{1,n-1}\}) \leq \qec(K_{1,m-1}\{K_{1,n-1}\})\]
	Again, by Theorem \ref{ulbound},
	\begin{equation}\label{twoinequality}
		\qec(P_{mn})<\qec(P_m\{K_{1,n-1}\})\ \ \  \mbox{and}\ \ \  \qec(K_{1,m-1}\{K_{1,n-1}\}) < \qec(K_{1,mn-1}).
	\end{equation}
	This concludes the proof.
\end{proof}

\section*{Acknowledgments}
We thank Apoorva Khare for a detailed reading of an earlier draft and for providing valuable feedback. P.N. Choudhury was partially supported by INSPIRE Faculty Fellowship research grant DST/INSPIRE/04/2021/002620
(DST, Govt.~of India), and IIT Gandhinagar Internal Project: IP/IITGN/MATH/PNC/2223/25. R. Nandi was supported by IIT Gandhinagar Post-Doctoral Fellowship IP/IITGN/MATH/PC/2223/20.


\begin{thebibliography}{99}
\bibitem{aouchiche}
M. Aouchiche and P. Hansen, {\it Distance spectra of graphs: a survey}, Linear Algebra Appl. {\bf 458}, 301-386 (2014).


\bibitem{bapat}
R.B. Bapat, {\it Graphs and Matrices}, Springer, Hindustan Book Agency,
New Delhi, 2010.

\bibitem{baskoro}
E.T. Baskoro, N. Obata, {\it Determining finite connected graphs along the quadratic embedding constants of paths}. Electron. J. Graph Theory Appl. {\bf 9}, 539-560 (2021).
\bibitem{mbozejko}
M. Bozejko, {\it Positive-definite kernels, length functions on groups and non-
	commutative von Neumann inequality.} Studia Math. {\bf 95}, 107-118 (1989).

\bibitem{bozejko}
M. Bozejko, T. Januszkiewicz and R.J. Spatzier, {\it Infinite Coxeter groups do not have Kazhdan's property.} J. Operator Theory {\bf 19}, 63-67 (1988).




\bibitem{graham}
R.L. Graham and P.M. Winkler, {\it On isometric embedding of graphs.} Trans. Am. Math. Soc. {\bf 288}, 527-536 (1985).
 
 \bibitem{Graham-Pollak}
 R.L.~Graham and H.O.~Pollak, {\it On the addressing problem for loop switching.} Bell Sys.\ Tech.\ J. {\bf 50}, 2495-2519 (1971).






\bibitem{haagerup}
U. Haagerup, {\it An example of a nonnuclear C*-algebra which has the metric
approximation property.} Invent. Math. {\bf 50}, 289-293 (1979).

\bibitem{hora}
A. Hora and N. Obata, {\it Quantum Probability and Spectral Analysis of
Graphs}, Springer, 2007.

\bibitem{horn}
R.A. Horn and C.R. Johnson, {\it Matrix analysis.} Second edition. Cambridge University Press, Cambridge, 2013.


\bibitem{lou}
Z. Lou, N. Obata and Q. Huang, {\it Quadratic embedding constants of graph joins.} Graphs Combin. {\bf 38} (2022), no. 5, Paper No. 161, 22 pp.



\bibitem{wmlotkowski}
W. Młotkowski, {\it Quadratic embedding constants of path graphs.} Linear Algebra Appl. {\bf 644}, 95-107
(2022).

\bibitem{mo}
W. Młotkowski and N. Obata, {\it On quadratic embedding constants of star product graphs.} Hokkaido
Math. J. {\bf 49}, 129-163 (2020).

\bibitem{bata}
N. Obata, {\it Complete multipartite graphs of non-QE class,} 2022. https://arxiv.org/abs/2206.05848.



\bibitem{nobata}
N. Obata, {\it Markov product of positive definite kernels and applications to
	Q-matrices of graph products.} Colloq. Math. {\bf 122}, 177-184 (2011).

\bibitem{obata}
N. Obata, {\it Positive Q-matrices of graphs.} Studia Math. {\bf 179}, 81-97 (2007).


\bibitem{zakiyyah}
N. Obata, A.Y. Zakiyyah, {\it Distance matrices and quadratic embedding of graphs.} Electron. J. Graph
Theory Appl. {\bf 6}, 37-60 (2018).



\bibitem{ijschoenberg}
I.J. Schoenberg, {\it Metric spaces and positive definite functions}. Trans. Amer.
Math. Soc. {\bf 44}, 522-536 (1938).

\bibitem{schoenberg}
I.J. Schoenberg, {\it Remarks to Maurice Frechet's article ``Sur la definition axiomatique d'une classe d'espace distancis vectoriellement applicable sur
l'espace de Hilbert"}. Ann. of Math. {\bf 36}, 724-732 (1935).



\bibitem{zhang}
H. Zhang, Y. Yang, and C. Li, {\it Kirchhoff index of composite graphs.} Discrete Appl. Math., {\bf 157}, 2918-2927 (2009).


\end{thebibliography}
\end{document}